\documentclass[]{amsart}
\usepackage{amsmath}
\usepackage{amssymb}
\usepackage[english]{babel}
\usepackage[utf8]{inputenc}
\usepackage{amsfonts}
\usepackage{amsthm}
\usepackage{color}
\usepackage{graphicx}
\usepackage{esint}
\numberwithin{equation}{section}

\newtheorem{thm}{Theorem}[section]
\newtheorem{prop}[thm]{Proposition}
\newtheorem{cor}[thm]{Corollary}
\newtheorem{lem}[thm]{Lemma}
\theoremstyle{remark}
\newtheorem{rmk}[thm]{Remark}
\theoremstyle{definition}
\newtheorem{defn}{Definition}[section]

\newtheorem{exmp}{Example}[section]

\DeclareMathOperator{\N}{\mathbb{N}}
\DeclareMathOperator{\R}{\mathbb{R}}
\DeclareMathOperator{\cF}{\mathcal{F}}

\begin{document}
\title{Preduals and double preduals of some Banach spaces}

\author[L. D'Onofrio]{Luigi D'Onofrio}
\address{Dipartimento di Scienze e Tecnologie, Universit\'a di Napoli \lq\lq Parthenope \rq\rq}
\email{donofrio@uniparthenope.it}

\author[G. Manzo]{Gianluigi Manzo}
\address{Dipartimento di Matematica e Applicazioni  ``R. Caccioppoli", Universit\`a degli Studi di Napoli Fe\-de\-rico II, via Cintia, 80126 Napoli, ITALY}
\email{gianluigi.manzo@unina.it}

\author[C. Sbordone]{Carlo Sbordone}
\address{Dipartimento di Matematica e Applicazioni ``R. Caccioppoli", Universit\`a degli Studi di Napoli Fe\-de\-rico II, via Cintia, 80126 Napoli, ITALY}
\email{sbordone@unina.it}

\author[R. Schiattarella]{Roberta Schiattarella}
\address{Dipartimento di Matematica e Applicazioni  ``R. Caccioppoli", Universit\`a degli Studi di Napoli Fe\-de\-rico II, via Cintia, 80126 Napoli, ITALY}
\email{roberta.schiattarella@unina.it}

\begin{abstract}
In this paper we study ways to establish when a Banach space can be identified as the dual or the double dual of another Banach space. To obtain these results, we relate these spaces with other, concrete Banach spaces - tipically $\ell^1$ and $\ell^\infty$ - and show that under suitable assumptions we can transfer properties of these spaces to the space we consider. In particular, we show how these results can be used to obtain in a simple way interesting results about spaces such as $BMO$ and $BV$.
\end{abstract}\maketitle
\begin{center}
Dedicated to the memory of Professor E. Vesentini.
\end{center}

\section{Introduction}
A Banach space $E$ is \emph{a dual space} if there exists a Banach space $E_*$ (a predual of $E$) such that $(E_*)^*$ is $E$. We mean that  $E_*$ is an isometric predual:  $$(E_*)^*\cong E,$$ where $\cong$ denotes isometric isomorphisms. We will denote isomorphism of Banach spaces by $\simeq$.\\ 
When $E$ is a dual space, we can define the weak-star topology, and the Banach-Alaoglu-Bourbaki Theorem says that the unit ball $\mathbb B_E:=\{x\in E:\|x\|_E\le 1\}$ of $E$ is weak-star compact.\\
This fact can be reversed, that is, assume that there exists a subset $\cF$  of $E^*$ that separates the points of $E$ and $\mathbb  B_E$ is $\sigma(E,\cF)$-compact (where $\sigma(E,\cF)$ is the coarsest topology on $E$ such that the elements of $\cF$ are continuous), then  the closed linear span of $\cF$
$$
cl_{E^{\star}}(span \cF)
$$
 defines a Banach space $E_*$ which is an isometric predual of $E$; in particular, the norm of $x\in E$ as a functional on $E_*$ is equal to its norm as an element of $E$
\begin{equation}\label{2}
\| x\|_{E}= \sup_{\| \varphi\|_{E_*}\leq 1} |\langle  x, \varphi \rangle|.
\end{equation}
(see \cite{kaijser1978note}, \cite{dixmier1948theoreme}).

This approach which relies on the existence of suitable $\cF\subset E^{*}$ is natural once we observe that any predual $E_*$ of a $E$ should be  viewed as a suitable closed subspace of the dual space $E^*$ (see \cite{Rossi}).
Theorem \ref{main} characterizes duals  $E$ of \emph{separable} spaces by simply requiring that $\mathcal F \subset E^{*}$ is \emph{countable}. In Section 4 we will complement this result with an explicit description of elements of $E_*$ by atomic decomposition. 

Let us now recall that for any subspace $U$ of a Banach space $Z$ we canonically have $U^*\cong Z^*/U^\perp$ and $(Z/U)^*\cong U^\perp$, where $U^{\perp}$ is the annihilator of $U$ in $Z^{*}$ (see for instance \cite{laxFunctional}, Exercise 2 and 3, pag. 76-77) . Theorem \ref{predualquotient} says that for any predual $P$ of $E$ there exists a Banach space $Z$, a continuous embedding $V: E\rightarrow Z^{*}$ with $E\cong V(E)$ weak star closed, such that $P\cong Z/\, ^{\perp}( V(E))$, i.e.
$$
\left(Z/\, ^{\perp} (V(E))\right)^{*} \cong E,
$$
where $^{\perp} V(E)$ is the subset of $Z$ of elements annihilating $V(E)$.

Two important examples, that we will consider in greater detail later, are $E=BMO(Q_0)$ the space of functions with bounded mean oscillation and $E=BV(Q_0)$ the space of functions of bounded variation ($Q_0=[0,1]^d\subset\R^d$). \\
To prove that $E$ is a dual space,  we denote by $I$ a set of indices with the cardinality $\kappa$ of the unit ball of $E_{*}$, $\mathbb B_{E_{*}}=(\psi_i)_{i\in I}$ and consider the space
	\begin{equation*}
	\ell^1_\kappa=\left\{\xi=(\xi_i)_{i\in I}\in I^{\R}:\|\xi\|_{\ell^1_\kappa}=\sum_{i\in I}|\xi_i|<\infty\right\}
	\end{equation*}
	whose dual is
	\begin{equation*}
	\ell^\infty_\kappa=\left\{\xi=(\xi_i)_{i\in I}\in I^{\R}:\|\xi\|_{\ell^\infty_\kappa}=\sup_{i\in I}|\xi_i|<\infty\right\}.
	\end{equation*}
Consider the continuous embedding
$$
V:E\rightarrow \ell^{\infty}_\kappa
$$
defined for $x\in E$ by
	\begin{equation*}
		V(x)= \{\langle \psi_i, x\rangle_{E}\}_{i\in I}\in\ell^\infty_\kappa
	\end{equation*}
	then, $E$ has an isometric predual, $$E_*= \ell^1_\kappa/^\perp (V(E))$$
	(see Theorem \ref{predualquotient}).\\
	Another interesting question is to see when $E_*$ has a predual; in this case, we will say that $E$ has \emph{double predual}. Let $\cF=\left\{  f_n \right\}_{n\in \mathbb N}\subset E^*$ be a countable, norming subset for $E$, thus $E$ is a Banach space under the norm
	$$
	\| x\|_{E}= \sup_{n\in \mathbb N} |\langle f_n, x  \rangle_{E}|.
	$$
	Now, we fix an isometric predual $C$ of $\ell^1$, hence $C\subset \ell^\infty$, and consider the isometric embedding
	$$
V:E\rightarrow \ell^{\infty}, \,\,V x(n)=\langle f_n, x \rangle_E
$$
and the set
$$
E_C=\left\{x\in E:V(x)\in C \right\}
$$
then, $(E_C)^{*}\subset E_{*}$ (see Remark \ref{inclusion}) and by Theorem \ref{duality} we have the isometric isomorphism
$$
(E_C)^{*} \cong E_{*}
$$
if (and only if) the (AP) condition
\begin{equation}\tag{AP}
		\forall x\in E,\;\;\exists\{y_j\}_{j\in\N}\subset E_C:y_j\to x \text{ in } \sigma(E, \cF),\quad\sup_{j\in\N}\|y_j\|_E<+\infty.
	\end{equation}
holds true.\\

Notice that for $E=BMO$, we have $E_{*}=\mathcal H^1$ the Hardy space, and for a suitable $C$, which can be obtained from Section \ref{preduals}, we recover $E_C=VMO$ and $BMO$ is double dual of $VMO$, where $VMO$ is the space of functions of vanishing mean oscillation.\\
 In case $E=BV(Q_0)$, $E_{*}$ is isomorphic to $$\left\{\varphi+\mathrm{div}\phi:\varphi\in C_0(Q_0),\phi\in C_0(Q_0;\R^d)\right\}$$
  (see \cite{AmbrosioFP}, \cite{FUSCO20181370}), but it is not possible to find $C$ such that $
(E_C)^{*} \cong E_{*}$, at least in the case that $C$ is M-embedded (see Section 2 and Theorem \ref{notunique}).\\ 
 
Another interesting case is the space $B$ of Bourgain--Brezis--Mironescu introduced in \cite{bourgain2015new} (see also \cite{ABBF}) to obtain a more general space than $VMO$ having the property that all integer valued functions are actually constant. This space is similar in nature to $BMO$, and enjoys similar properties.

\section*{Acknowledgments}
\noindent
The authors are members of the Gruppo Nazionale per l'Analisi Matematica, la Probabilit\`a e le loro Applicazioni (GNAMPA) of the Istituto Nazionale di Alta Matematica (INdAM). The research of L. D. has been funded by \lq\lq~Sostegno alla Ricerca Locale~\rq\rq , Universit\'a degli studi di Napoli \lq\lq~Parthenope~\rq\rq. The research of R.S. has been funded by PRIN Project 2017JFFHSH.

\section{Preliminaries}
We fix some notation.

For a Banach space $E$, we say that a subset $\mathcal F\subset E^*$ is \textit{a norming set } for $E$ if 
\begin{equation}\label{eq:Fnorming}
\|x\|_E=\sup_{f\in \mathcal F}|\langle f,x\rangle_E|
\end{equation}
 for all $ x\in E$.
Let $E$ be a Banach space; consider the following sequence spaces.

Let
$$
\ell^p(E)= \left\{(a_n): \mathbb N \to E: \sum_{n=1}^{\infty} |a_n|^p< \infty   \right\}
$$
where $1\leq p<\infty$. Then $\ell^p$ is  a vector space on which the standard norm is defined by
$$
\| (a_n)\|_{\ell^p(E)}= \left(\sum_{n=1}^{\infty} |a_n|^p \right)^\frac{1}{p}.
$$

Let 
$\ell^\infty(E)$ be the vector space of all bounded sequences in $E$ equipped  with the following norm:
$$
\|(a_n)\|_{\ell^\infty(E)}= \sup_{n\in \mathbb N} |a_n|.
$$

An important linear subspace of $\ell^\infty(E)$ is the space $c_0(E)$ 
of sequences having norm converging to zero. $(c_0(E), \|\cdot \|_{\ell^\infty(E)})$ is a normed vector space.

 The space of bounded finitely additive functions on $\mathbb N$ is denoted by $\text{ba}(\mathbb N)$ \cite{yosida1952finitely}:
 
 \begin{equation*}
 	ba(\N)=\left\{m:\mathcal{P}(\N)\to\R\mid\forall A,B\subset\N,A\cap B=\emptyset, m(A\cup B)=m(A)+m(B)\right\}.
 \end{equation*}

It is a known result of the theory of Banach spaces that if $U$ is a subspace of the Banach space $Z$ then there is a natural isomorphism between $(Z/U)^*$ and the annihilator $U^\perp$ of $U$ in $Z^*$,
$$
U^\perp = \left\{ f\in Z^*: \langle f, u\rangle =0\,\, \,\, \forall u \in U  \right\},
$$
which is given by the map that sends $f\in U^\perp$ into $\tilde{f}:z+U\in Z/U\mapsto f(z)$.\\
There is also a natural isomorphism between $U^*$ and $Z^*/(U^\perp)$, which is obtained similarly.\\
If $Y$ is a subspace of $Z^*$, we denote by $^\perp Y$ the subset of $Z$ of elements annihilating $Y$.

\begin{defn}\cite{harmand2006m}
   A Banach space $E$ is said to be \emph{$M$-embedded} if the natural decomposition
   $$
   E^{***}=E^* \oplus_1 E^\perp
   $$
   is an $\ell^1$ direct sum.
   
   A Banach space $E$ is said to be \emph{$L$-embedded} if there exists a closed subspace $F\subset E^{**}$ such that 
   $$
   E^{**}=E \oplus_1 F
   $$
   is an $\ell^1$ direct sum.
\end{defn}

 
\section{Results}\label{maintheory}

\begin{thm}\label{predualquotient}
	Let $E$ be a Banach space. Then $E$ is the isometric dual of a Banach space $E_*$ if and only if there exists a Banach space $Z$ and a continuous embedding $V:E\to Z^*$ such that $V$ is an isometric isomorphism between $E$ and $V(E)$ and $V(E)$ is weak-star closed in $Z^*$. In this case, the space $E_*=Z/(^\perp V(E))$ is such an isometric predual.
\end{thm}
\begin{proof}
	Assume there is a map $V$. Then we have that $(^\perp V(E))^\perp=V(E)$, so that $\left(Z/(^\perp V(E))\right)^*\cong V(E)\cong E$.\\
	Conversely, let $E_*$ be an isometric predual of $E$.\\
	We now take $\kappa$ equal to the cardinality of $\mathbb B_{E_*}=\{\psi_i\}_{i\in I}$ and define the map
	\begin{equation*}
	V: x\in E\mapsto \{\langle \psi_i,x\rangle_{E}\}_{i\in I}\in\ell^\infty_\kappa.
	\end{equation*}
	The map $V$ is obviously an isomorphism, and it is possible to show that $V(E)$ is weak-star closed in $\ell^\infty_\kappa$. By the Krein-Smulian theorem, it is enough to show that $V(E)\cap \mathbb B_{\ell^\infty_\kappa}$ is weak-star closed.\\
	Let $\{x_\lambda\}_{\lambda}$ be a net in the unit ball of $E$, $\mathbb B_E$, such that $V(x_\lambda)$ converges weak star to $\xi=\{\xi_i\}_{i\in I}\in\ell^\infty_\kappa$. Since $\mathbb B_E$ is $\sigma(E,E_*)$-compact, we can assume without loss of generality that $\langle \psi_i,x_\lambda\rangle_E\to \langle \psi_i,x\rangle_E\quad\forall i\in I$ for some $x\in \mathbb B_E$. But since $\langle\xi,\delta^i\rangle_{\ell^1_\kappa}=\xi_i$, where $\delta^i$ is the element of $\ell^1_\kappa$ equal to $1$ at index $i$ and $0$ elsewhere, we have that $\langle \psi_i,x_\lambda\rangle_E\to\xi_i$ for all $i\in I$, so $\xi_i=\langle \psi_i,x\rangle_E$, i.e. $\xi=V(x)$, which proves our claim.\\
	To end the proof, we should prove that 
	$$E_*= \ell^1_\kappa/^\perp V(E).$$
	
	We notice that $\langle \psi_i,x\rangle_E=\langle \delta_i,V(x)\rangle_{\ell^\infty}$, so that $\psi_i$ coincides with the equivalence class of $\delta_i$, and since $\ell^1_\kappa/^\perp V(E)=\text{cl}_{E^*}(\text{span}\{[\delta^i]\}_{i\in I})$ we prove this result as well.
\end{proof}
We use previous result to ricognize that the space
\begin{equation}\label{BVexample0}
BV(Q_0)=\left\{u\in L^1(Q_0):V[u]<\infty\right\},
\end{equation}
where 
$$
V[u]:=\sup\left\{\displaystyle\int_{Q_0}u\,\mathrm{div}\varphi\, dx:\varphi\in C_0^\infty(Q_0),\|\varphi\|_\infty\le 1\right\}
$$
endowed with the norm $\|u\|_{BV}:=\|u\|_{L^1}+V[u]$ is a dual.\\
Choose 
$$
Z= C_0(Q_0)^{d+1}
$$
we have
$$
Z^*=\left\{ \mu= (\mu_0, \mu_1, \cdots, \mu_d): \mu_j \in \mathcal M(Q_0) \right\}
$$
where $\mathcal M(Q_0)$ denotes the space of Radon measures on $Q_0$.\par
Let us also consider the subspace $Y$ of $Z$ which is the closure of the subspace
$$
T=\left\{ \phi \in Z: \phi= (\phi_0, \phi_1, \cdots \phi_d) \in C^\infty_0 (Q_0)^d: \phi_0= \sum_{i=1}^d \frac{\partial \phi_i}{\partial x_i}\right\}.
$$
Consider the map
$$
V: u\in E \to \left( u \mathcal L^d, D_1u, \cdots, D_d u \right)\in Z^*
$$
and notice that (see \cite{AmbrosioFP} Remark 3.12)
$$
\| u\|_{BV} \leq 2 \| Vu\|_{Z^*} \leq 2 \| u\|_{BV}.
$$
Moreover
$$
V(E) \text{ is weakly closed in } Z^*.
$$
Hence, by Theorem \ref{predualquotient}, we obtain that 
$BV$ is dual:
\begin{equation}\label{BVexample1}
BV\cong \left( Z/ ^\perp V(E)\right)^*.
\end{equation}
Actually $BV_*$ is a separable space.\\

It is possible to show that we can replace in some sense $\mathbb B_{E_*}$ with $\cF$ in the proof of  previous theorem,  as long as $\mathbb B_E$ is compact with respect to the topology $\sigma(E,\cF)$ \cite{dixmier1948theoreme,kaijser1978note}.\\
In case $\cF=\{f_n\}_{n\in\N}$ is a sequence, the resulting predual $E_*$ is separable. We now show that the converse is true, and in particular it can be used to characterize duals of separable spaces.
\begin{thm}\label{main}
	$E$ is the dual of a separable space if and only if there exists a countable, norming subset $\cF\subset E^*$ such that $\mathbb B_E$ is $\sigma(E,\cF)$-compact.
\end{thm}
\begin{proof}
	Let us assume that $E$ is the dual of a separable space $E_*$, let $D$ be a dense countable set in $E_*$ and take $\cF:=\left\{f\in E^*\mid\exists z\in D\backslash\{0\}:f=\displaystyle\frac{z}{\|z\|_{E_*}}\right\}$. $\cF$ is obviously countable, norming and $\mathbb B_E$ is $\sigma(E,\cF)$-compact because it is weak-star compact.
	
	To prove that $E$ is the dual of a separable space, let $\mathcal F$ be a norming sequence $\{f_n\}_n$ of elements of $E^*$ and $V:x\in E\mapsto \{\langle f_n,x\rangle\}_{n\in\N}\in\ell^\infty$.  The fact that $V$ is an isometry is trivial by \eqref{eq:Fnorming}.
	
	Recall that $V(E)^{\perp}$ is the subspace of $\mathrm{ba}(\N)$ of elements annihilating $V(E)$:
$$V(E)^{\perp}=\{m\in\mathrm{ba}(\N):\langle m,V(x)\rangle_{\ell^\infty}=0\;\;\,\forall x\in E\}.$$
Since the Banach space $\ell^1$ is naturally understood as the subspace of $\text{ba}(\mathbb N)$ consisting of the countably additive measures, we denote by  $^\perp V(E)$ the subset of $\ell^1$ of elements annihilating $V(E)$:
$$
^\perp V(E)=\{\xi\in\ell^1:\langle V(x),\xi\rangle_{\ell^1}=0\;\; \forall x\in E\}=V(E)^\perp\cap\ell^1.
$$
	
	 We show that $E$ has an isometric predual $E_*$
	$$E_*= \ell^1/^\perp V(E).$$
	As before, we can prove that $V(E)$ is weak-star closed in $\ell^\infty$, and that
	$$E_*= \ell^1/^\perp V(E)\cong  \text{cl}_{E^\star}\left(\text{span}\mathcal F  \right).$$
	
\end{proof}

If we take $E=BMO(Q_0)$, the space of functions with bounded mean oscillation (see \cite{john1961functions}), where $Q_0$ is the unit cube in $\R^d$, we know that its separable predual is the Hardy space $\mathcal H^1$. See Example \ref{BMO} for details. As a consequence, one can construct a norming set $\cF$ for $BMO(Q_0)$ based on the atomic decomposition of $\mathcal H^1$ (see Section \ref{atomic}). An alternative construction is described in Section \ref{examples}, along with details on $BMO(Q_0)$ and $B(Q_0)$.

Our aim is now to give conditions under which $E_*$ \emph{ is dual of an another Banach space.}

Notice that in the case of Example \ref{BMO}, $E_*$ is also dual of the separable Banach space $VMO$, the space of functions with vanishing mean oscillation (see \cite{sarason1975functions}).

From now on, with $\cF=\left\{ f_n\right\}_{n\in \mathbb N}\subset E^*$ we will denote a countable, norming subset  for $E$ generating a separable predual $E_*$ of $E$.\\
Let 
$$
V: x\in E \to \left\{\langle f_n, x\rangle_{E} \right\}_{n\in \mathbb N}\in \ell^\infty
$$
the continuous embedding of $E$ and let us consider an isometric predual $C$ of $\ell^1$ and define $E_C(\cF)=E_C=\{x\in E:V(x)\in C\}$. 
We want to show that under an appropriate condition we have that $(E_C)^*\cong E_*$.
\begin{thm}\label{duality}
Let $E$ be a Banach space and $E_*$ its predual $E_*\simeq cl_{E^{\star}}(span \cF)$. Fix $C$ an isometric predual of $\ell^1$. Then, the following condition holds:
	\begin{equation}\tag{AP}
		\forall x\in E,\;\;\exists\{y_j\}_{j\in\N}\subset E_C:y_j\to x\,\, \text{ in } \sigma( E, \mathcal F),\quad\sup_{j\in\N}\|y_j\|_E<+\infty.
	\end{equation}
	if, and only if, $(E_C)^*\cong E_*$ via the duality pairing $\langle\psi,y\rangle_{E_C}=\langle y,\psi\rangle_{E_*}$.
\end{thm}
\begin{proof}
	We start from the fact that $V$ maps $E_C$ isometrically in $C$, so that $(E_C)^*\cong \ell^1/(V(E_C))^\perp$. Using Theorem \ref{predualquotient}, it is sufficient to show that $$(V(E_C))^\perp=(V(E))^\perp\cap\ell^1$$ (note that $(V(E_C))^\perp$ is taken in $\ell^1$, while $(V(E))^\perp$ is taken in $\mathrm{ba}(\N)$). Equivalently, for $\xi\in\ell^1$ we need to have that if $\xi$ annihilates $V(E_C)$ then it also annihilates $V(E)$. To show this we fix $\xi\in (V(E_C))^\perp$ and $x\in E$, and we take a sequence $\{y_j\}_{j\in\N}\subset E_C$ as in condition (AP). Identifying elements of $\ell^1$ as elements of $\mathrm{ba}(\N)$, we can write
	\begin{equation*}
		\langle \xi,V(x)\rangle_{\ell^\infty}=\langle \xi,V(y_j)\rangle_C+\langle \xi,V(x-y_j)\rangle_{\ell^\infty}=\langle \xi,V(x-y_j)\rangle_{\ell^\infty}=\sum_{n\in \mathbb N}\xi_n\langle f_n,x-y_j\rangle_E.
	\end{equation*}
	Fix $\varepsilon>0$. For $k\in\N$ have that \begin{equation*}
	\left|\sum_{n=k}^{\infty}\xi_n\langle f_n,x-y_j\rangle_E\right|\le\sup_{n\in\N}\left|\langle f_n,x-y_j\rangle_E\right|\sum_{n=k}^{\infty}|\xi_n|,
	\end{equation*}
	which goes to $0$ as $k\to\infty$ since
	$$
	\sup_{n\in\N}\left|\langle f_n,x-y_j\rangle_E\right|\le\sup_{n\in\N}\left|\langle f_n,y_j\rangle_E\right|+\sup_{n\in\N}\left|\langle f_n,x\rangle_E\right|=\|y_j\|_E+\|x\|_E<\infty,
	$$
	so we can take $N=N(\varepsilon)\in\N$ such that
	\begin{equation*}
	\left|\sum_{n=N+1}^{\infty}\xi_n\langle f_n,x-y_j\rangle_E\right|\le\frac{\varepsilon}{2}.
	\end{equation*}
	Now, since for all $n\in\N$ we have $\langle f_n,y_j\rangle_E\to\langle f_n,x\rangle_E$, there exists a $\nu\in\N$ such that for all $j>\nu$ we have
	\begin{equation*}
		\left|\sum_{n=1}^{N}\xi_n\langle f_n,x-y_j\rangle_E\right|\le\frac{\varepsilon}{2},
	\end{equation*}
	so combining the previous inequalities we have that $|\langle \xi,V(x)\rangle_{\ell^\infty}|\le\varepsilon.$ To prove our claim it suffices to take $\varepsilon\to 0$.\\
	
	Conversely, since $(E_C)^{**}\cong E$ and the canonical isometry $J:E_C\to E$ coincides with the immersion map, the Goldstine theorem implies that the unit ball of $E_C$ is $\sigma(E,E_*)$ dense in the unit ball of $E$, or equivalently that for every $x\in E$ there exists a sequence $\{y_j\}_{j\in\N}\subset E_C$ such that $y_j\overset{*}{\rightharpoonup}x$ and $\sup_{j\in\N}\|y_j\|_{E_C}\le\|x\|_E$. Since $\cF\subset E_*$, the claim is proven.
\end{proof}
\begin{rmk}\label{inclusion}
	In case (AP) does not hold, we still have that $^\perp V(E)\subset (V(E_C))^\perp$, therefore $(E_C)^*\cong \ell^1/(V(E_C))^\perp\subset \ell^1/^\perp V(E)\cong E_*$.
\end{rmk}


When (AP) holds, $E_C$ has other interesting properties, under additional conditions on $C$.
\begin{prop}\label{Mid}
	Suppose $C$ is an M-embedded predual of $\ell^1$ and that condition (AP) holds. Then $E_C$ is an $M$-embedded Banach space, i.e. we have the decomposition $(E_C)^{***}=E^*=(E_C)^*\oplus_1 (E_C)^\perp$.
\end{prop}
\begin{proof}
	Let $\phi\in E^*=(E_C)^{***}$. From the canonical decomposition $(E_C)^{***}=(E_C)^*\oplus (E_C)^\perp$, we can write $f=\psi+g$, with $\psi\in (E_C)^*$ and $g\in (E_C)^\perp$. Viewing $f$ as a continuous linear operator on $V(E)\subset\ell^\infty$, we can use the Hahn-Banach theorem to extend it to an element $m\in\mathrm{ba}(\N)$ such that $\|m\|_{\mathrm{ba}(\N)}=\|f\|_{E^*}$. 
	We now use the known fact that $C$ is an $M$-embedded space, i.e. we have a decomposition $(\ell^1)^{**}\cong\mathrm{ba}(\N)=\ell^1\oplus_1 C^\perp$. Using this decomposition we write $m=\xi+m_s$, with $\xi\in\ell^1$, $m_s\in C^\perp$ and $\|m\|_{\mathrm{ba}(\N)}=\|\xi\|_{\ell^1}+\|m_s\|_{\mathrm{ba}(\N)}$.\\
	We notice that $\psi=0$ iff $\xi\in (V(E))^\perp$: if $x\in E$ we consider a sequence $y_j$ as in (AP) and we obtain, reasoning as in the proof of theorem \ref{duality}, that \begin{equation*}
	\begin{split}
		\langle \xi, V(x)\rangle_C&=\lim\limits_{j\to\infty}\langle\xi,V(y_j)\rangle_C=\lim\limits_{j\to\infty}\langle m,V(y_j)\rangle_{\ell^\infty}\\
		&=\lim\limits_{j\to\infty}\langle f,y_j\rangle_E=\lim\limits_{j\to\infty}\langle \psi,y_j\rangle_E=\langle \psi,x\rangle_E.
		\end{split}
	\end{equation*}
	A consequence of this is that $\psi$ can be identified with the equivalence class of $\xi$, and $g$ with the equivalence class of $m_s$.\\
	We have now
	\begin{equation*}
		\|f\|_{E^*}=\|m\|_{\mathrm{ba}(\N)}=\|\xi\|_{\ell^1}+\|m_s\|_{\mathrm{ba}(\N)}\ge\|\psi\|_{(E_C)^*}+\|g\|_{(E_C)^\perp},
	\end{equation*}
	which combined with the triangular inequality $\|f\|_{E^*}\le\|\psi\|_{E^*}+\|g\|_{E^*}=\|\psi\|_{E_C^*}+\|g\|_{(E_C)^\perp}$ concludes the proof.
\end{proof}
Recall that a Banach space $Y$ is a unique (isometric) predual if for all Banach spaces $Z$, $Y^*\cong Z^*$ implies $Y\cong Z$, while it is a strongly unique (isometric) predual if the only norm one projection from $Y^{***}$ to $Y^*$ with weakly closed kernel is the projection induced by the canonical decomposition $Y^{***}=Y^*\oplus Y^\perp$; equivalently, if $P\subset Y^{**}$ is such that $P^*\cong Y^*$ via the duality induced by $Y^*$, then $P$ must coincide with $Y$, since for a closed subspace $C$ of a Banach space $X$ we have $^\perp(C^\perp)=C$.\\
It is known \cite{harmand2006m} that if $Z$ is an $M$-embedded Banach space then $Z^*$ is a strongly unique predual, while $Z$ is not, unless it is reflexive. We remark that it was also shown \cite{pfitzner2007separable} that a separable $L$-embedded Banach space is a unique predual, which partially generalizes the previously mentioned result since the dual of an $M$-embedded space is $L$-embedded.
\begin{cor}
	Suppose $C$ is an M-embedded predual of $\ell^1$ and that condition (AP) holds. Then $E_*$ is a strongly unique predual, while $E_C$ is not, unless it is reflexive.
\end{cor}
\begin{prop} 
Let $\cF= \left\{f_j\right\}_{j\in \mathbb N}\subset E^*$. Suppose $C$ is an M-embedded predual of $\ell^1$ and that condition (AP) holds. Then for $x\in E$ the following distance formula holds:
	\begin{equation}
	\mathrm{dist}_E(x,E_C)=\limsup_{j\to\infty}|\langle f_j,x\rangle_E|.
	\end{equation}
\end{prop}
\begin{proof}
	The inequality $\mathrm{dist}_E(x,E_C)\ge\limsup_{j\to\infty}|\langle f_j,x\rangle_E|$ is trivial, so let us now prove the other inequality. To do this, we can write $\mathrm{dist}_E(x,E_C)=\sup\limits_{\substack{f\in E_C^\perp\\\|f\|=1}}\left|\langle f,x\rangle_E\right|$. Lifting $f$ to an element $m$ of $\mathrm{ba}(\N)$ we can then write
	\begin{equation*}
	\begin{split}
	\mathrm{dist}_E(x,E_C)&\le\sup\limits_{\substack{m\in \left(\ell^1\right)^\perp\\\|m\|=1}}\left|\langle m,V(x)\rangle_{\ell^\infty}\right|\\
	&=\sup\limits_{\substack{m\in \left(\ell^1\right)^\perp\\\|m\|=1}}\left|\langle m,(V(x))^{(j)}\rangle_{\ell^\infty}\right|=\limsup_{j\to\infty}\left|\langle f_j,x\rangle_E\right|,
	\end{split}
	\end{equation*}
	where $(V(x))^{(j)}$ is the sequence that is equal to $0$ for all indices less or equal than $j$ and to $(V(x))_k=\langle f_k,x\rangle_E$ when the index $k$ is greater than $n$. This concludes the proof.
\end{proof}

\section{Atomic decompositions}\label{atomic}

The space $BMO(\mathbb R^d)$ of functions of bounded mean oscillation of John-- Nirenberg \cite{john1961functions} is a \emph{dual space}. A simple proof can be given for $d=1$ if we equipe $E=BMO(\mathbb R)$ with the norm 
\begin{equation}\label{4.1}
\| u\| = \sup_{I \subset \mathbb R} \left( \fint_I |u- u_I|^2\,  dx  \right)^{\frac{1}{2}}
\end{equation}
where $I$ is a bounded interval, $u_I=\fint_I u$ and $u\in L^2_{\text{loc}}(\mathbb R)$.

The famous Theorem of C. Fefferman \cite{fefferman} more precisely and deeply says that $BMO(\mathbb R^d)$ is the \emph{dual} of the Hardy space $\mathcal H^1(\mathbb R^d)$ of $L^1$ functions whose vector valued Riesz transform $Rf=(R_1f, \cdots, R_d f):= D I_1 f$ belongs to $L^1$, where $I_1f$ is the Riesz potential
\begin{equation}\label{4.2}
I_1f(x)= \frac{1}{\gamma}\int_{\mathbb R^d} \frac{f(y)}{|x-y|^{d-1}}\, dy
\end{equation}
This is the original definition of real Hardy spaces $\mathcal H^1(\mathbb R^d)$ due to Stein-- Weiss \cite{SW}.

The norm in $\mathcal H^1$ is given by
\begin{equation}\label{4.3}
\| f\|_{\mathcal H^1}= \| f\|_{L^1} + \sum_{j=1}^d \| R_j f\|_{L^1}
\end{equation}
and the pairing between BMO and $\mathcal H^1$ is represented by
\begin{equation}
f\in \mathcal H^1 \to \int u f\, dx \qquad u\in BMO.
\end{equation}
Actually this integral converges absolutely if $f$ belongs to a certain \emph{dense} linear subspace $\mathcal H^1_a$ of $\mathcal H^1$ consisting of finite linear combinations of atoms (see \cite{S}, Ch.3, Section 2.4).

By  \emph{atom} in $\mathcal H^1(\mathbb R)$ we mean here a function $a: \mathbb R\to \mathbb R$ such that there exists an interval $I$ such that
\begin{equation}\label{4.5}
\begin{cases} 
a\in L^2_{\text{loc}}(\mathbb R), \text{supp } a \subset I, \\
\\
\fint_I a=0,\;\, \| a\|_{L^2(I)}\leq \dfrac{1}{|I|^{1/2}}.
\end{cases}
\end{equation}
Let us now assume that $BMO(\mathbb R)$ is dual of the Banach space $E_*$ and let us show that an element $f$ of $E_*$ has the form
\begin{equation}\label{4.6}
f= \sum_{j=1}^{\infty}\lambda_j a_j, \sum |\lambda_j|<\infty, \, a_j\in \mathcal A
\end{equation}
where $\mathcal A$ is the set of atoms \eqref{4.5}.

Observe first of all that for all intervals $I$
$$
\left( \fint_I |u- u_I|^2\,  dx  \right)^{\frac{1}{2}}= \sup_{a\in \mathcal A_I} \big| \fint_I u a \big|
$$
where $a\in \mathcal A_I$ iff $\fint_I a=0$, $\fint_I a^2\leq 1$.

Hence
$$
\| u\| =\sup \big| \int_I u a \big|
$$
where the supremum is taken over all atoms verifying \eqref{4.5}. It follows that elements $f$ of the predual $E_*$ of $E$ have representation \eqref{4.6}.
Moreover, if $(a_j)$ is a sequence of $\mathcal H^1_a$ and $(\lambda_j)$ of real numbers satisfying $\sum |\lambda_j|<\infty$, then the series 
$$
f= \sum \lambda_j a_j
$$
converges in the sense of distributions and its sum belongs to $\mathcal H^1$ with $\|f\|_{\mathcal H^1} \leq c \sum |\lambda_j|$ (see \cite{S}, Ch. 3, Section 2.2.).

Also the space $BV$ of functions of bounded variation is a dual of a separable Banach space with atomic decomposition. See \cite{d2020atomic} for a general result on preduals of Banach spaces defined and normed by the fact that $x\in E$ if and only if
$$
\|x\|_E= \sup_{L\in \mathcal L} \| Lx\|_Y<\infty
$$
where $\mathcal L$ is a collection of operators $L: X\to Y$, $X$ and $Y$ Banach spaces.

More precisely, if $X$ is a reflexive Banach space and $\mathcal L=(L_j)_{j\in \mathbb N}$
is a sequence of linear forms $L_j\subset X^*$ and 
$$
E=\left\{ x\in X: \sup_j |L_j x|<\infty \right\}
$$
is a Banach space under the norm
$$
\|x\|_E= \sup_{j} | L_j x|
$$
and $E\hookrightarrow X$ then $E$ has an isometric predual with atomic decomposition.
The cases $E=BMO$, $E= BV$ or $E=B$ the space recently introduced by Bourgain-Brezis-Mironescu are included (see also \cite{angrisaniascionemanzo}) .\\
In \cite{d2020DSS} the same atomic decomposition result is proved weakening the reflexivity assumption on $X$.

This result on atomic decomposition can be generalized in a very simple way using Theorem \ref{main}.
	Let $E_*$ be a separable Banach space $\cF=\{f_j\}_{j\in\N}$, be as in the first part of the proof of Theorem \ref{main}, and apply the theorem to $E:=(E_*)^*$ and $\cF$. This allows us to identify $E_*$ with $\ell^1/^\perp V(E)$, so to every $\psi\in E_*$ we can associate $\xi=\{\xi_j\}_{j\in\N}\in\ell^1$, defined up to elements of $^\perp V(E)$, such that
	\begin{equation}
		\psi\simeq [\xi]=\left[\sum_{j\in\N}\xi_j\delta_j\right]=\sum_{j\in\N}\xi_j[\delta_j]=\sum_{j\in N}\xi_j f_j.
	\end{equation}
	This representation of the elements of $E_*$ has the nice property that, by definition of quotient,
	\begin{equation*}
		\|\psi\|_{E_*}=\inf\|\xi\|_{\ell^1},
	\end{equation*}
where the infimum is taken between all such representation. This is actually better than the usual results on atomic decompositions, where there is only an equivalence of norms. 

\section{Properties of some preduals of $\ell^1$}\label{preduals}
We know that $c_0$, the space of all sequences of real numbers converging to zero, equipped with the supremum norm, is an isometric predual of $\ell^1$ and is M-embedded, so one could hope that it could be used for various examples. However, it is far too limiting to only consider this space, and to show some interesting examples of spaces satisfying this structure we need to introduce and study particular preduals of $\ell^1$.
\begin{defn}
	Let $X$ be a locally compact Hausdorff space. We define by $C_0(X)$ the set of continuous real valued functions $f:X\to\R$ on $X$ that vanish at infinity, i.e. for every $\varepsilon>0$ there exists a compact $K_\varepsilon$ such that $|f(x)|<\varepsilon$ for every $x\notin K_\varepsilon$.
\end{defn}
Let $\tau$ be a locally compact Hausdorff topology on $\N$. We denote the space $C_0(\N,\tau)$ by $c_0(\tau)$. The following theorem holds
\begin{thm}
	Let $\tau$ be a locally compact Hausdorff topology on $\N$. Then $(c_0(\tau))^*\cong\ell^1$.
\end{thm}
\begin{proof}
	 The Riesz-Markov-Kakutani representation theorem for $C_0$ \cite{rudin2006real} implies that every element $L$ of the dual of $c_0(\tau)$ is uniquely represented in the form $L_\mu:x=\{x_j\}_{j\in\N}\in c_0(\tau)\mapsto\displaystyle\int_{\N} x_j\,d\mu(j)$ for some regular signed Borel measure $\mu$, and the norm of $\mu$ in $(c_0(\tau))^*$ is equal to its total variation $|\mu|(\N)$.\\
	 For $j\in\N$, we define $a_j=\mu(\{j\})=L_\mu(\delta_j)$, where $\delta_j$ is the element of $c_0(\tau)$ that is equal to $1$ for the index $j$ and $0$ for every other index. It is easy to see that $\sum_{j\in\N}a_j\le|\mu|(\N)<\infty$, so that $a=\{a_j\}_{j\in\N}\in\ell^1$, and since the span of the $\delta_j$ is dense in $c_0(\tau)$ we can actually identify $\mu$ with $a\in\ell^1$, and $|\mu|(\N)$ actually coincides with $\|a\|_{\ell^1}$, concluding the proof.
\end{proof}
We now prove the following.
\begin{prop}
	Let $\tau$ be a locally compact Hausdorff topology on $\N$. Then $c_0(\tau)$ is M-embedded.
\end{prop}
\begin{proof}
	 We have the decomposition $ba(\N)=\ell^1\oplus (c_0(\tau))^\perp$, where $(c_0(\tau))^\perp$ is the subset of $ba(\N)$ consisting of those functions that vanish on all compacts of $\tau$. We need to show that it is an $\ell^1$ decomposition, i.e. for every decomposition $m=\xi\delta+m_s$ of $m\in ba(\N)$, where $\delta$ indicates the counting measure on $\N$, $\xi=\{\xi_j\}_{j\in\N}\in\ell^1$ and $m_s\in (c_0(\tau))^\perp$, we have $\|m\|_{ba(\N)}=\|\xi\|_{\ell^1}+\|m_s\|_{ba(\N)}$.\\
	Without loss of generality, we can assume that $m$, $\xi$ and $m_s$ are non-negative, since we can decompose every element $\mu\in ba(\N)$ in its positive and negative part $\mu=\mu_+ -\mu_-$, and we have $\|\mu\|_{ba(\N)}=\|\mu_+\|_{ba(\N)}+\|\mu_-\|_{ba(\N)}=\mu_+(\N)+\mu_-(\N)$ \cite{yosida1952finitely}. It is also not hard to show that if $\mu\in\ell^1$ ($\mu\in c_0(\tau)^\perp$ respectively) then $\mu_+$ and $\mu_-$ are also in $\ell^1$ ($c_0(\tau)^\perp$ respectively).\\
	Since for every $\tau$-compact $K$ we have $m_s(K)=0$, we can write
	\begin{align*}
		\|m\|_{ba(\N)}&=m(\N)=m(K)+m(\N\backslash K)\\
		&=\sum_{j\in K}\xi_j+\sum_{j\in \N\backslash K}\xi_j+m_s(\N\backslash K)\\
		&=\sum_{j\in K}\xi_j+\sum_{j\in \N\backslash K}\xi_j+\|m_s\|_{ba(\N)}.
	\end{align*}
	Now, for all $\varepsilon>0$ there exists a finite subset $K$ of $\alpha$ such that $\sum_{j\in K}\xi_j\ge\|\xi\|_{\ell^1}-\varepsilon$, and since all finite subsets of $\N$ are compact for $\tau$, by taking $\varepsilon\to 0$ we obtain $\|m\|_{ba(\N)}=\|\xi\|_{\ell^1}+\|m_s\|_{ba(\N)}$, concluding the proof.
\end{proof}
We can use this to obtain the following result, which will be useful to obtain examples.
\begin{prop}\label{littleo}
	Let $E$ be a Banach space, $\widetilde{\cF}\subset E^*$ a family of linear forms on $E$ equipped with a locally compact, Hausdorff and separable topology $\tilde{\tau}$. Assume that $\widetilde{\cF}$ is norming for $E$, that $\mathbb B_E$ is $\sigma(E,\widetilde{\cF})$-compact and that the maps $f\in\cF\mapsto\langle f,x\rangle_E$ are $\tilde{\tau}$-continuous for all $x\in E$.\\
	Let $E_0=\left\{y\in E:\displaystyle\limsup_{\substack{f\in\widetilde{\cF}\\f\to\infty}}\left|\langle f,y\rangle_E\right|=0\right\}$, where $f\to\infty$ is intended in the one point compactification of $\tilde{\tau}$. Then $E_0=E_{c_0(\tau)}(\cF)$, where $\cF$ is a countable dense subset of $\widetilde{\cF}$ and $\tau$ is the topology induced by $\tilde{\tau}$ on $\cF$.
\end{prop}
\begin{proof}
	We start by noticing that $\mathbb B_E$ is $\sigma(E,\cF)$-compact, $\tau$ is locally compact and Hausdorff, and by continuity we also have that $\cF$ is norming for $E$. We can thus apply the theory of section \ref{maintheory}, and we easily get that $E_0=E_{c_0(\tau)}(\cF)$.
\end{proof}
\section{examples}\label{examples}
	We start by stating a lemma, which will help prove that $\mathbb B_E$ is $\sigma(E,\cF)$-compact in many cases.
	\begin{lem}\label{refl}
		Let $X$ be a reflexive Banach space, $\cF\subset X^*$ and $E=\{x\in X:\|x\|_E:=\sup_{f\in\cF}|\langle f,x\rangle_X|<\infty\}$. Assume $E$ is continuously embedded in $X$. Then $\mathbb B_E$ is $\sigma(E,\cF)$-compact.
	\end{lem}
	\begin{proof}
		Without loss of generality, we assume that $E$ is densely contained in $X$. Since $E\hookrightarrow X$, there exists $\lambda>0$ such that $\mathbb B_E\subset \lambda \mathbb B_X$, which implies that $\mathbb B_E$ has weakly compact closure in $X$. But $\mathbb B_E$ is closed in $X$: for every net $\{x_\alpha\}_{\alpha\in A}\subset \mathbb B_E$ converging to $\tilde{x}\in X$ and for every $f\in\cF$ we have $\langle f,x_\alpha\rangle_X\to\langle f,\tilde{x}\rangle_X$, which implies that $|\langle f,\tilde{x}\rangle_X|\le 1$, and since $|\langle f,x\rangle_X|<\infty$ if and only if $x\in E$ we have $\tilde{x}\in \mathbb B_E$, which implies that $\mathbb B_E$ is weakly compact in $X$, and in particular $\sigma(E,\cF)$-compact in $E$.
	\end{proof}
	Let us denote by $Q_0$ the unit cube $[0,1]^d$ in $\R^d$.
\begin{exmp}\label{BMO}
	Consider the space \cite{john1961functions}
	\begin{equation*}
		BMO(Q_0)=\left\{u\in L^1(Q_0)/\R:[u]_*:=\sup\limits_{Q\subset Q_0}\fint_Q|u(x)-u_Q|dx<\infty\right\},
	\end{equation*}
	where $Q$ varies between all cubes contained in $Q_0$ and $u_I=\displaystyle\fint_I u(x)dx:=\displaystyle\frac{1}{|I|}\int_I u(x)dx$.\\
	Let $\widetilde{\mathcal Q}$ be the family of all cubes contained in $Q_0$ and $\mathcal{Q}$ the subset of rational cubes. We can consider the bijection between $\widetilde{\mathcal Q}$ and $\overset{\circ}{Q_0}\times (0,1]$: if $s(x)$, for $x\in\overset{\circ}{Q_0}$, defines the maximum side length of a cube centered in $x$, we associate to $(x_0,t)\in\overset{\circ}{Q_0}\times(0,1]$ the cube $q(x_0,t)$ centered in $x_0$ with side length $s(x_0)t$. This map induces a topology on $\widetilde{\mathcal Q}$ for which $\mathcal{Q}$ is dense\\
	Let $\mathbb B=\mathbb B_{L^\infty(Q_0)}$ and consider the family $\tilde{\cF}=\{\phi_{Q,h}\}_{I\in\mathcal Q,h\in \mathbb B}$, where $\phi_{Q,h}$ is defined as
$$ \phi_{Q,h}: u\in BMO(Q_0)/\R\mapsto\fint_Q (u(x)-u_Q)h(x)dx\in\R.$$
Let us endow $\mathbb B$ with the weak topology on $L^2(Q_0)$: this choice makes $\mathbb B$ Hausdorff and separable, and since $\mathbb B$ is closed and bounded in $L^2$, it is also compact. As a consequence, we can endow $\widetilde{\cF}$ with the topology induced by the topologies on $\widetilde{\mathcal Q}$ and $\mathbb B$.
	
	We now verify the conditions to apply Proposition \ref{littleo}.\\
	\textbf{1. $\widetilde{\cF}$ is norming for $BMO(Q_0)$.} This is just a consequence of the fact that $\|f\|_{L^1}=\displaystyle\sup_{\substack{g\in L^\infty\\\|g\|_{L^\infty}\le 1}}\displaystyle\int fg$ and the definition of $BMO(Q_0)$.\\
	\textbf{2. The unit ball of $BMO(Q_0)$ is $\sigma(BMO(Q_0),\widetilde{\cF})$-compact.} This can be shown by applying Lemma \ref{refl} with $X=L^2(Q_0)/\R$. In this case $X^*=L^2_0(Q_0)$, the space of $L^2(Q_0)$ functions having zero integral on $Q_0$. Seeing $\phi_{Q,h}$ as acting on $L^1(Q_0)/\R$, we have the canonical identification $\phi_{Q,h}\sim\frac{\chi_Q}{|Q|}h-\int_Q h\,dx$, which is easily seen to be in $L^2(Q_0)$.\\
	\textbf{3. The map $\phi\mapsto\langle\phi,u\rangle_{BMO(Q_0)}$ is continuous for all $u\in BMO(Q_0)$.} We first note that the map $Q\mapsto |Q|$ is continuous, and for fixed $u\in L^2(Q_0)$ we have 
		$$\left|\int_{Q_1}u(x)\,dx-\int_{Q_2}u(x)\,dx\right|\le\int_{Q_1\Delta Q_2}|u(x)|\,dx\le|Q_1\Delta Q_2|^{1/2}\|u\|_{L^2(Q_0)},$$

	so that the map $Q\mapsto\int_Q u(x)\,dx$, and as a consequence the map $Q\mapsto u_Q$, is continuous. Let us now fix $u\in BMO(Q_0)\subset L^2(Q_0)/\R$, $Q_1,Q_2\in\mathcal{Q}$, $h_1,h_2\in \mathbb B_{L^\infty(Q_0)}$ and compute:
	
\begin{equation*}
\begin{split}
	&\left|\langle\phi_{Q_1,h_1}-\phi_{Q_2,h_2},u\rangle_{BMO(Q_0)}\right|\\
	&=\left|\fint_{Q_1} (u(x)-u_{Q_1})h_1(x)dx-\fint_{Q_2} (u(x)-u_{Q_2})h_2(x)dx\right|\\
	&\le\left|\fint_{Q_1} (u(x)-u_{Q_1})h_1(x)dx-\fint_{Q_1} (u(x)-u_{Q_1})h_2(x)dx\right|\\
	&+\left|\fint_{Q_1} (u(x)-u_{Q_1})h_2(x)dx-\fint_{Q_1} (u(x)-u_{Q_2})h_2(x)dx\right|\\
	&+\left|\fint_{Q_1} (u(x)-u_{Q_2})h_2(x)dx-\fint_{Q_2} (u(x)-u_{Q_2})h_2(x))dx\right|\\
	&=\left|\fint_{Q_1} (u(x)-u_{Q_1})(h_1(x)-h_2(x))dx\right|+|u_{Q_1}-u_{Q_2}|\left|\fint_{Q_1}h_2(x)dx\right|\\
	&+\left|\fint_{Q_1} u(x)h_2(x)dx-\fint_{Q_2} u(x)h_2(x)dx\right|\\
	&\le \left|\fint_{Q_1} (u(x)-u_{Q_1})(h_1(x)-h_2(x))dx\right|+|u_{Q_1}-u_{Q_2}|+|[uh_2]_{Q_1}-[uh_2]_{Q_2}||,
	\end{split}
	\end{equation*}
	
	where $u(x)-u_{Q_1}\in L^2(Q_0)$ because of the John-Nirenberg inequality, so that the first term goes to zero as $h_1-h_2\rightharpoonup 0$ in $L^2(Q_0)$, while the second and the third can be made arbitrarily small as the cubes $Q_1$ and $Q_2$ are close in the topology on $\mathcal{Q}$; as a result of this, our claim is proven.\\
	We thus obtain that
	$$
	E_0=VMO(Q_0):=\left\{v\in BMO(Q_0):\limsup_{|I|\to 0}\fint_I|v(x)-v_I|dx=0\right\}
	$$
	is the \lq\lq little o \rq\rq to $BMO(Q_0)$ (for the theory of o-O spaces see \cite{perfekt2015weak}, \cite{perfekt2017m},\cite{ascione2020structure}). Using the fact that $\text{cl}_{BMO(Q_0)}{C^\infty(Q_0)}=VMO(Q_0)$ (\cite{Garnett}, Chapter VI), one can show the (AP) property by convolution with suitable functions (\cite{d2020atomic}), and with that we obtain many classically known results about $VMO(Q_0)$ (\cite{sarason1975functions}, \cite{perfekt2013duality}) such as the M-embedding of this space and the distance formula
	\begin{equation*}
		\mathrm{dist}_{BMO(Q_0)}(u,C^\infty(Q_0))=\limsup_{|I|\to 0}\fint_I|u(x)-u_I|dx.
	\end{equation*}
\end{exmp}

\begin{rmk}
In conclusion, following \cite{CW}, we notice that $\mathcal H^1$ is one of the few examples of a separable, non reflexive space which is a dual space.
\end{rmk}

\begin{exmp}
	The space $B(Q_0)$, introduced in \cite{bourgain2015new} (see also \cite{ABBF}), is defined as
	\begin{equation}
		B(Q_0)=\left\{u\in L^1(Q_0):[u]:=\sup_{0<\varepsilon\le 1}[u]_\varepsilon<\infty\right\},
	\end{equation}
	where $[u]_\varepsilon=\sum_{\mathcal P\subset\mathcal F_\varepsilon}\left(\varepsilon^{d-1}\sum_{Q\in\mathcal P}\fint_Q|u(x)-u_Q|dx\right)$ and $\mathcal F_\varepsilon$ is the set of all unions of at most $\varepsilon^{1-d}$ disjoint cubes of side length $\varepsilon$. It is associated with the closed subspace
	$$B_0(Q_0):=\left\{u\in B(Q_0):\limsup_{\varepsilon\to 0}[u]_\varepsilon=0\right\},$$
	which can be regarded as the \lq\lq little o\rq\rq space corresponding to $B$. In \cite{bourgain2015new} it was proven, among other things, that $(B_0(Q_0))^{**}\cong B(Q_0)$, and additional properties were shown in \cite{d2020atomic}, including a description of the predual of $B$ via atomic decomposition.\\
	It is not hard to extend the constrution for $BMO$ case to $B$, and the results from sections \ref{maintheory} and \ref{preduals} hold with little changes.
\end{exmp}

\begin{exmp}
	Let $(K,d)$ be a compact metric space. We define the Lipschitz space
	\begin{equation*}
		\mathrm{Lip}(K,d)=\left\{u\in C(K):[u]_{\mathrm{Lip}(K,d)}:=\sup_{x\ne y\in K}\frac{|u(x)-u(y)|}{d(x,y)}<\infty\right\}
	\end{equation*}
	and its subspace
	\begin{equation*}
		\mathrm{lip}(K,d)=\left\{v\in\mathrm{Lip}(K,d):\limsup_{d(x,y)\to 0}\frac{|v(x)-v(y)|}{d(x,y)}=0\right\}.
	\end{equation*}
	Particular cases of this spaces are the H\"older spaces $C^{0,r}(K)=\mathrm{Lip}(K,d_{\mathrm{eucl}^r})$ with $K$ a compact subset of $\R^n$ and $0<r\le 1$, with their vanishing subsets $c^{0,r}(K)$. We remark that $c^{0,1}(K)=\mathrm{lip}(K)\equiv\R$.\\
	The quantity $[\cdot]_{\mathrm{Lip}(K,d)}$ is a seminorm on $\mathrm{Lip}(K,d)$, and there are many equivalent ways to turn it into a norm: in particular, we choose to restrict $\mathrm{Lip}(K,d)$ to the space of functions vanishing in a fixed point $P$. The resulting space is a Banach space.\\
	Let $K_s$ be a countable dense of subset of $K$ (it exists because compact metric spaces are totally bounded) and consider $\cF=\left\{\phi_{x,y}=\frac{\delta_x-\delta_y}{d(x,y)}:x,y\in K_s,x\ne y\right\}$, where $\delta_x(u)=u(x)$. $\cF$ is obviously norming for $\mathrm{Lip}(K,d)$, and it can be easily proven that $B_{\mathrm{Lip}(K,d)}$ is $\sigma(\mathrm{Lip}(K,d),\cF)$-compact: it is enough to show that $B_{\mathrm{Lip}(K,d)}$ is $\sigma(\mathrm{Lip}(K,d),\Delta)$-compact, where $\Delta=\{\delta_x:x\in K\}$, but $\sigma(\mathrm{Lip}(K,d),\Delta)$-convergence is just pointwise convergence, so by applying the Ascoli theorem to $\mathbb B_{\mathrm{Lip}(K,d)}$ we easily prove the claim.\\
	To apply Proposition \ref{littleo}, we just give $\cF$ the topology induced by $K^2\backslash\text{diag}(K^2)$; with this topology, the assumptions can be easily shown. Using this construction, the corresponding little space is $\mathrm{lip}(K,d)$.\\
	There is a wide literature that studied conditions equivalent to the property $(\mathrm{lip})^{**}\cong \mathrm{Lip}$ \cite{kalton2004spaces,hanin1994isometric}: in particular, for a compact space it turns out that those conditions are equivalent to the AP property, and that distances of the form $d^r$ with $0<r<1$ have this property. In this case we obtain the results about M-embedding and the distance formula. This can thus be seen as a generalization of \cite{angrisani2019duality} (see \cite{ambrosio2016linear} for connection with Kantorovich-Rubinstein transport distance) .
\end{exmp}
\begin{exmp}\label{BVexample}
	Let us go back to the space $BV(Q_0)$.\\
	For every $(d+1)$-tuple $(\varphi_{0},\varphi_{1},\dots,\varphi_{d})\in C_0(Q_0)^{d+1}$ we can define the linear action
	\begin{equation*}
	u\in BV(Q_0)\mapsto\int_{Q_0}u\varphi_{0}dx+\sum_{i=1}^{d}\int_{Q_0}u\frac{\partial\varphi_{i}}{\partial x_i}dx,
	\end{equation*}
	which is just another way to obtain the construction in \eqref{BVexample0}.\\
	Let now $\Phi=\{\varphi_j\}_{j\in\N}$ be a countable dense subset of $\{\varphi\in C_c^\infty(Q_0):\|\varphi\|_\infty\le 1\}$, and consider $\cF:=\Phi^{d+1}$. From the definition itself of $BV(Q_0)$, $\cF$ is norming, and since $W(Q_0):=\text{cl}_{(BV(Q_0)^*)}{span(\cF)}=\{\varphi+\mathrm{div}\phi:\varphi\in C_0(Q_0),\phi\in C_0(Q_0;\R^d)\}$ is precisely the (vector) space obtained in \eqref{BVexample1} by construction and it automatically satisfies the assumptions of Theorem \ref{predualquotient}.\\
	However, Proposition \ref{Mid} does not apply to $BV(Q_0)$; more specifically, it is possible to prove the following:
	\begin{thm}\label{notunique}
		The space $W(Q_0)$ is not a strongly unique predual.
	\end{thm}
	\begin{proof}
		We will show that there exists a closed subspace $P\subset (BV(Q_0))^*$, not coinciding with $W(Q_0)$, such that $P^*\cong BV(Q_0)$ under the duality pairing induced by $\langle\cdot,\cdot\rangle_{BV(Q_0)}$.\\
		We start by noting that $C(Q_0)$, endowed with the supremum norm, is contained isometrically in $(BV(Q_0))^*$: more specifically, for $\Psi\in C(Q_0)$, we can consider the functional
		\begin{equation*}
			u\in BV(Q_0)\mapsto \int_{Q_0} \Psi u\,dx.
		\end{equation*}
		It is easy to show that $\left|\displaystyle\int_{Q_0}\Psi u\,dx\right|\le\|\Psi\|_{C(Q_0)}\|u\|_{L^1(Q_0)}\le\|\Psi\|_{C(Q_0)}\|u\|_{BV(Q_0)}$, so that $\Psi\in (BV(Q_0))^*$, and by testing on constant functions you get $\|\Psi\|_{(BV(Q_0))^*}=\|\Psi\|_{C(Q_0)}$. Moreover, we get that this identification is injective. Similarly, $W(Q_0)$ contains an isometric copy of $C_0(Q_0)$, corresponding to $C_0(Q_0)\times \{0\}^d$.\\
		Now, $C_0(Q_0)$ contains an isometric copy of $c_0$: if $\{K_n\}_{n\in\N}$ is an exhaustion by compacts of $Q_0$, take $S_1:=K_1$, $S_n:=K_n\backslash K_{n-1}$ and $\sigma_n\in C_0(S_n)$, $\|\sigma_n\|_{C_0(S_n)}=1$. We can then consider the map
		\begin{equation*}
			\rho:\xi=\{\xi_n\}_{n\in\N}\in \ell^1\mapsto\sum_{n\in N}\xi_n\sigma_n\in C(Q_0),
		\end{equation*}
		which can be easily seen to be an isometry, and $\rho(c_0)\subset C_0(Q_0)$. A theorem by Sobczyk \cite{sobczyk1941projection} says that if a separable Banach space contains an isometric copy of $c_0$ then it is complemented: in particular, we can write $W(Q_0)=W'\oplus\rho(c_0)$. If we now consider $C$ a predual of $\ell^1$ not isomorphic to $c_0$, then $W'\oplus\rho(C)$, with the norm induced by $(BV(Q_0))^*$ (since $\rho(C)\subset\rho(\ell^\infty)\subset C(Q_0)\subset (BV(Q_0))^*$), is a predual of $BV(Q_0)$ via the canonical duality pairing, i.e. $W(Q_0)$ is not a strongly unique predual.
	\end{proof}
\end{exmp}


\begin{thebibliography}{10}

\bibitem{ABBF}
L.~Ambrosio, J.~Bourgain, H.~Brezis, and A.~Figalli.
\newblock Bmo-type norms related to the perimeter of sets.
\newblock {\em Comm. Pure Appl. Math.,}, 69:1062--1086, 2016.

\bibitem{AmbrosioFP}
L.~Ambrosio, N.~Fusco, and D.~Pallara.
\newblock {\em Functions of bounded variation and free discontinuity problems}.
\newblock Oxford Mathematical Monographs. The Clarendon Press, Oxford
  University Press, New York, 2000.

\bibitem{ambrosio2016linear}
L.~Ambrosio and D.~Puglisi.
\newblock Linear extension operators between spaces of {L}ipschitz maps and
  optimal transport.
\newblock {\em Journal f{\"u}r die reine und angewandte Mathematik (Crelles
  Journal)}, 764.

\bibitem{angrisani2019duality}
F.~Angrisani, G.~Ascione, L.~D'Onofrio, and G.~Manzo.
\newblock Duality and distance formulas in {L}ipschitz-{H}\"older spaces.
\newblock {\em Atti Accad. Naz. Lincei Rend. Lincei Mat. Appl}, 31(2):401--419,
  2020.

\bibitem{angrisaniascionemanzo}
F.~Angrisani, G.~Ascione, and G.~Manzo.
\newblock Atomic decomposition of finite signed measures on compacts of
  {$\mathbb R^n$}.
\newblock {\em to appear on Ann. Acad. Sci. Fenn. Math}, 2020.

\bibitem{ascione2020structure}
G.~Ascione and G.~Manzo.
\newblock o--{O} structure of some rearrangement invariant {B}anach function
  spaces.
\newblock {\em Journal of Elliptic and Parabolic Equations}, 1--23, 2020.

\bibitem{bourgain2015new}
J.~Bourgain, H.~Brezis, and P.~Mironescu.
\newblock A new function space and applications.
\newblock {\em J. Eur. Math. Soc.(JEMS) 6}, 17(9):2083--2101, 2015.

\bibitem{CW}
R.R. Coifman and G.~Weiss.
\newblock Extensions of {H}ardy spaces and their use in analysis.
\newblock {\em Bull. Amer. Math. Soc.}, 83:569--645, 1977.

\bibitem{dixmier1948theoreme}
J.~Dixmier.
\newblock Sur un th{\'e}or{\`e}me de {B}anach.
\newblock {\em Duke Mathematical Journal}, 15(4):1057--1071, 1948.

\bibitem{d2020atomic}
L.~D'Onofrio, L.~Greco, K.~M. Perfekt, C.~Sbordone, and R.~Schiattarella.
\newblock Atomic decompositions, two stars theorems, and distances for the
  {B}ourgain--{B}rezis--{M}ironescu space and other big spaces.
\newblock {\em Ann.I. H. Poincar{\'e} -AN}, 37:653--661, 2020.

\bibitem{d2020DSS}
L.~D’Onofrio, C.~Sbordone, and R.~Schiattarella.
\newblock Atomic decomposition for preduals of some {B}anach spaces.
\newblock {\em preprint}, 2020.

\bibitem{fefferman}
C.~Fefferman.
\newblock Characterizations of bounded mean oscillation.
\newblock {\em Bull. A. M. S.}, 77:587--588, 1971.

\bibitem{FUSCO20181370}
N.~Fusco and D.~Spector.
\newblock A remark on an integral characterization of the dual of {BV}.
\newblock {\em Journal of Mathematical Analysis and Applications}, 457(2):1370
  -- 1375, 2018.

\bibitem{Garnett}
J.~Garnett.
\newblock {\em Bounded analytic functions}, volume 236 of {\em Graduate Texts
  in Mathematics}.

\bibitem{hanin1994isometric}
L.~Hanin.
\newblock On isometric isomorphism between the second dual to the “small”
  {L}ipschitz space and the “big” {L}ipschitz space.
\newblock 316--324, 1994.

\bibitem{harmand2006m}
P.~Harmand, D.~Werner, and W.~Werner.
\newblock {\em {M}-ideals in {B}anach spaces and {B}anach algebras}.
\newblock Springer, 2006.

\bibitem{john1961functions}
F.~John and L.~Nirenberg.
\newblock On functions of bounded mean oscillation.
\newblock {\em Communications on pure and applied Mathematics}, 14(3):415--426,
  1961.

\bibitem{kaijser1978note}
S.~Kaijser.
\newblock A note on dual {B}anach spaces.
\newblock {\em Mathematica Scandinavica}, 325--330, 1978.

\bibitem{kalton2004spaces}
N.~J. Kalton.
\newblock Spaces of lipscitz and {H}{\"o}lder functions and their applications.
\newblock {\em Collectanea Mathematica}, 171--217, 2004.

\bibitem{laxFunctional}
P.D. Lax.
\newblock {\em Functional analysis}.
\newblock Pure and Applied Mathematics (New York), 2002.

\bibitem{perfekt2013duality}
K.~M. Perfekt.
\newblock Duality and distance formulas in spaces defined by means of
  oscillation.
\newblock {\em Arkiv f{\"o}r matematik}, 51(2):345--361, 2013.

\bibitem{perfekt2015weak}
K.~M. Perfekt.
\newblock Weak compactness of operators acting on o--{O} type spaces.
\newblock {\em Bulletin of the London Mathematical Society}, 47(4):677--685,
  2015.

\bibitem{perfekt2017m}
K.~M. Perfekt.
\newblock On {M}-ideals and o--{O} type spaces.
\newblock {\em Mathematica Scandinavica}, 102(1):151--160, 2017.

\bibitem{pfitzner2007separable}
H.~Pfitzner.
\newblock Separable {L}-embedded {B}anach spaces are unique preduals.
\newblock {\em Bulletin of the London Mathematical Society}, 39(6):1039--1044,
  2007.

\bibitem{Rossi}
S.~Rossi.
\newblock A characterization of separable conjugate spaces.
\newblock {\em arXiv:1102.4325}.

\bibitem{rudin2006real}
W.~Rudin.
\newblock {\em Real and complex analysis}.
\newblock Tata McGraw-hill education, 2006.

\bibitem{sarason1975functions}
D.~Sarason.
\newblock Functions of vanishing mean oscillation.
\newblock {\em Transactions of the American Mathematical Society},
  207:391--405, 1975.

\bibitem{sobczyk1941projection}
A.~Sobczyk.
\newblock Projection of the space $(m)$ on its subspace $(c_0)$.
\newblock {\em Bulletin of the American Mathematical Society}, 47(12):938--947,
  1941.

\bibitem{S}
E.~Stein.
\newblock {\em Harmonic analysis: real-variable methods, orthogonality, and
  oscillatory integrals}.
\newblock Princeton University Press, Princeton, NJ, 1993.

\bibitem{SW}
E.M. Stein and G.~Weiss.
\newblock On the theory of harmonic functions of several variables: {I}. {T}he
  theory of $h^p$-spaces.
\newblock {\em Acta Math}, 103(1-2):25--62, 1960.

\bibitem{yosida1952finitely}
K.~Yosida and E.~Hewitt.
\newblock Finitely additive measures.
\newblock {\em Transactions of the American Mathematical Society},
  72(1):46--66, 1952.

\end{thebibliography}
\end{document}